\documentclass{article}
\usepackage{amsmath,amssymb}
\usepackage{enumerate}

\newtheorem{theorem}{Theorem}[section]
\newtheorem{lemma}[theorem]{Lemma}
\newtheorem{proposition}[theorem]{Proposition}
\newtheorem{corollary}[theorem]{Corollary}

\newtheorem{remark}[theorem]{Remark}

\newcommand{\Z}{\mathbb{Z}}

\renewcommand{\ker}{\operatorname{Ker}}
\newcommand{\id}{\operatorname{id}}

\newcommand{\sym}{\operatorname{Sym}}
\newcommand{\aut}{\operatorname{Aut}}
\newcommand{\ret}{\operatorname{Ret}}
\newcommand{\soc}{\operatorname{Soc}}

\newenvironment{proof}{\par\noindent{\bf Proof.}}{$\qed$\par\bigskip}
\newcommand{\qed}{\enspace\vrule  height6pt  width4pt  depth2pt}
\usepackage{color}

\begin{document}

\title{A family of irretractable square-free solutions of the Yang-Baxter
equation}
\author{D. Bachiller \and F. Ced\'o \and E. Jespers \and J. Okni\'{n}ski}
\date{}

\maketitle

\begin{abstract} A new family of non-degenerate involutive
set-theoretic solutions of the Yang-Baxter equation is constructed.
All these solutions are strong twisted unions of multipermutation
solutions of multipermutation level at most two. A large subfamily
consists of irretractable and square-free solutions. This subfamily
includes a recent example of Vendramin \cite[Example~3.9]{Ven}, who
first gave a counterexample to Gateva-Ivanova's Strong Conjecture
\cite[Strong Conjecture 2.28(I)]{GI}. All the solutions in this
subfamily are new counterexamples to Gateva-Ivanova's Strong
Conjecture
 and also they answer a question
of Cameron and Gateva-Ivanova \cite[Open Questions~6.13
(II)(4)]{GIC}. It is proved that the natural left brace structure on
the permutation group of the solutions in this family has trivial
socle. Properties of the permutation group and of the structure
group associated to these solutions are also investigated. In
particular, it is proved that the structure groups of finite
solutions in this subfamily are not poly-(infinite cyclic) groups.
\end{abstract}



\section{Introduction}
Since its appearance in a paper of Yang \cite{Yang}, the Yang-Baxter
equation has become an important equation in mathematical physics
and also in quantum group theory. It has stimulated a lot of
activity and led to a diversity of new methods in several related
areas of algebra. Recall that a set-theoretic solution of the
Yang-Baxter equation on a non-empty set $X$  is a bijective map
$r\colon X\times X \longrightarrow X\times X$ such that
    $$r_{12}r_{23}r_{12} = r_{23}r_{12}r_{23},$$
where $r_{ij}$ denotes the map $X\times X \times X \longrightarrow
X\times X \times X$ acting as $r$ on the $(i,j)$ components and as
the identity on the remaining component.  Drinfeld, in
\cite{drinfeld} suggested that is of interest to study
set-theoretic solutions of the quantum Yang-Baxter equation
$$R_{12}R_{13}R_{23}=R_{23}R_{13}R_{12}.$$
It is known that if $\tau\colon X\times X\longrightarrow X\times
X$ is the twist map $\tau(x,y)=(y,x)$, then a map $r\colon X\times
X \longrightarrow X\times X$ is a set-theoretic solution of the
Yang-Baxter equation if and only if $R=\tau\circ r$ is a
set-theoretic solution of the quantum Yang-Baxter equation.

In recent years, a special class of solutions of this type, the
non-degenerate involutive solutions, has received a lot of attention
\cite{CJO, CJO2, CJR, ESS, GI, GIC, gat-maj, GIVdB, JO, JObook,
Rump1, Rump}. Also, this class of solutions has connections with
many topics in mathematics,  such as semigroups of $I$-type and
Bieberbach groups \cite{GIVdB}, bijective $1$-cocycles \cite{ESS},
radical rings \cite{Rump}, triply factorized groups \cite{sysak},
Hopf algebras \cite{EtingofGelaki}, regular subgroups of the
holomorf and Hopf-Galois extensions \cite{CR, FCC}, groups of
central type \cite{BenDavid, BDG}.

To study involutive non-degenerate set-theoretic solutions of the
Yang-Baxter equation, Rump introduced in \cite{Rump} a new algebraic
structure, called a brace. Recall that a left brace is a set $B$
with two binary operations, a sum $+$ and a product $\cdot$, such
that $(B,+)$ is an abelian group (the additive group of $B$),
$(B,\cdot)$ is a group (the multiplicative group of $B$) and
$$a\cdot (b+c)+a=a\cdot b+a\cdot c,$$
for all $a,b,c\in B$. Rump has begun to develop the theory of braces
in a series of papers \cite{Rump3, Rump2, Rump4, Rump5, Rump6,
Rump7}. The usefulness of this algebraic structure to solve problems
about this type of solutions of the Yang-Baxter equation is
confirmed by the results proven in \cite{CJO2}.  Even more, in
\cite{BCJ}, the classification of involutive non-degenerate
set-theoretic solutions of the Yang-Baxter equation is reduced to
the classification of left braces.  Rump \cite{Rump2} and Bachiller
\cite{B} classified some special classes of left braces. These
results indicate that the classification of arbitrary  left braces
(even in the finite case) seems to be a very difficult problem. If
$B$ is a finite left brace, then it is known that the multiplicative
group of $B$ is solvable \cite{ESS}. Using some preliminary ideas of
Rump, stated in \cite{Rump7}, and developing new ideas on left
braces,  it has recently been proven in \cite{B2} that there exist
finite $p$-groups which are not multiplicative groups of finite left
braces. This   answers  in the negative a question which appears
implicitly in \cite{ESS} and explicitly in \cite{CJR}. It is an open
problem to characterize the finite solvable groups which are
multiplicative groups of left braces.

A possible strategy to classify finite left braces is the following.
First, construct and classify the finite simple left braces. Second,
develop the theory of extensions of finite left braces.

Recall that an ideal of a left brace $B$ is a normal subgroup $I$ of
the multiplicative group of $B$ such that
$$ba-b\in I,$$
for all $b\in B$ and all $a\in I$. The socle of a left brace $B$ is
$$\soc(B)=\{ a\in B\mid ab=a+b\mbox{ for all }b\in B\}.$$
It is an ideal of $B$ (see \cite[page 107]{CJO2}). One says that the
left brace $B$ is simple if $B\neq \{ 1\}$ and $\{ 1\}$ and $B$ are
the only ideals of $B$. Rump, in \cite{Rump}, has shown that the
only simple finite nilpotent left braces (that is, the
multiplicative group of $B$ is nilpotent) are the cyclic groups
$\Z/(p)$, with $p$ a prime, and it turns out that the multiplication
of the brace is equal to the sum.   Recently Bachiller \cite{B3} has
developed a method to construct finite non-nilpotent simple left
braces and has given some families of such braces. To apply this
method of constructing new finite simple left braces it is important
to discover  new families of finite left braces with trivial socle.
Note that, obviously, any finite non-nilpotent simple left brace
should have trivial socle. Some families of finite left braces with
trivial socle have been given in \cite{CR,H}. The natural structure
of a left brace on the permutation group of a finite irretractable
solution yields a class of finite braces with trivial socle (see
Lemma~\ref{socle} below). Therefore, to find new families of finite
irretractable solutions, or, more general, new families of finite
left braces with trivial socle is of interest from the point of view
of the classification of finite left braces.

Another key ingredient of the classification would be the theory of
extensions of left braces. However, very little is known about this
(see \cite{B3,BDG2}).

Vendramin in \cite[Example~3.9]{Ven} gives a counterexample to a
conjecture of Gateva-Ivanova \cite[Strong Conjecture 2.28(I)]{GI},
see Section~\ref{constr}, by constructing  an irretractable
square-free involutive non-degenerate set-theoretic solution of the
Yang-Baxter equation $(X,r)$ with $|X|=8$. It is remarkable that
among the 2471 square-free non-degenerate involutive set-theoretic
solutions on a set $X$ with $|X|\leq 8$ this  is the only
counterexample to the Gateva-Ivanova conjecture (see Remark 3.11 in
\cite{Ven}). Furthermore, studying this example of Vendramin one can
check that it is a strong twisted union of two multipermutation
solutions of multipermutation level two. Thus this yields a negative
answer to a question posed by Cameron and Gateva-Ivanova \cite[Open
Questions~6.13 (II)(4)]{GIC}, although Vendramin did not notice this
fact in \cite{Ven}.

In this paper we construct a large family of irretractable
square-free involutive non-degenerate  solutions of the Yang-Baxter
equation that includes the example of Vendramin. Thus these
solutions are new counterexamples to \cite[Strong Conjecture
2.28(I)]{GI}. These solutions are strong twisted unions of
multipermutation solutions of multipermutation level $2$,
corresponding to their orbits under the action of its permutation
group. Hence, these solutions also yield a negative answer to a
question posed by Cameron and  Gateva-Ivanova in \cite[Open
Questions~6.13 (II)(4)]{GIC}. The natural structure of left brace on
the permutation group of these solutions provides a new family of
left braces with trivial socle. We also study another structure
associated to a solution of the Yang-Baxter equation, the so called
structure group (which is a solvable Bieberbach group  if the
solution is finite) introduced by Etingof, Schedler and Soloviev
\cite{ESS}. In particular, we prove that these groups are not
poly-(infinite cyclic). This is in contrast with the case of
multipermutation solutions, whose structure groups are always
poly-(infinite cyclic).

\section{Preliminary results}  \label{prelim}

We begin by recalling the necessary terminology and notation. Let
$X$ be a non-empty set and  $r\colon X\times X\longrightarrow
X\times X$ a map, and write $r(x,y)=(\sigma_x(y),\gamma_y(x))$.
Recall that $(X,r)$ is  said to be   a non-degenerate involutive
set-theoretic solution of the Yang-Baxter equation if and only if
the following  properties hold.
\begin{itemize}
\item[(1)] $r^2=\id_{X^2}$  ($r$ is involutive).
\item[(2)] $\sigma_x,\gamma_x\in \sym(X)$, for all $x\in X$  ($r$ is non-degenerate).
\item[(3)] $r_{12}r_{23}r_{12}=r_{23}r_{12}r_{23}$.
\end{itemize}

It is easy to check that (1) and (2) imply
$\gamma_y(x)=\sigma^{-1}_{\sigma_x(y)}(x)$, for all $x,y\in X$.
\bigskip

\noindent {\bf Convention.} By a solution of the YBE we mean a
non-degenerate involutive set-theoretic solution of the Yang-Baxter
equation.
\bigskip

A solution $(X,r)$ of the YBE is called square-free if
$r(x,x)=(x,x)$ for all $x\in X$.  If $r(x,y)=(y,x)$, i.e. if all
$\sigma_{x}=\id_{X}$, then $r$ is called the trivial solution.

The structure group of a solution $(X,r)$ of the YBE is the group
$G(X,r)=\langle X\mid xy=zt$ whenever $r(x,y)=(z,t)\rangle$. The
permutation group of $(X,r)$, denoted $\mathcal{G}(X,r)$,  is the
subgroup of the symmetric group $\sym(X)$  on  $X$ generated by $\{
\sigma_x\mid x\in X\}$.

Etingof, Schedler and Soloviev in \cite{ESS} proposed the following
interesting operator for  studying the structure group $G(X,r)$ and
to classify  solutions of the YBE. We recall its definition. Given a
solution $(X,r)$ of the YBE, with
$r(x,y)=(\sigma_x(y),\gamma_y(x))$, define the  equivalence relation
$\sim$ on $X$ by $$x\sim y\mbox{ if and only if }
\sigma_x=\sigma_y.$$ We denote by $\bar x$ the $\sim$-class of $x\in
X$. The retraction $\ret(X,r)$ of $(X,r)$ is the solution $(\bar
X,\bar r)$, where $\bar X=X/\sim$ and $\bar r(\bar x,\bar
y)=(\overline{\sigma_x(y)},\overline{\gamma_y(x)})$. A solution
$(X,r)$ of the YBE is said to be  a multipermutation solution if
there exists a positive integer $n$ such that $\ret^n(X,r)$ is a
solution on a set of cardinality $1$.  The multipermutation level of
a multipermutation solution $(X,r)$ of the YBE is the smallest
positive integer $n$ such that $\ret^n(X,r)$ is a solution on a set
of cardinality $1$.  One says that $(X,r)$ is irretractable if
$\ret(X,r)=(X,r)$.

Rump in \cite{Rump} introduced  a new algebraic structure, called a
left brace. This allows another possible strategy to attack the
problem of constructing and classifying the solutions of the YBE.
Recall that a left brace is a set $B$ with two operations, an
addition $+$ and a multiplication $\cdot$, such that $(B,+)$ is an
abelian group, $(B,\cdot)$ is a group and
$$a\cdot(b+c)+a=a\cdot b+a\cdot c,$$
for all $a,b,c\in B$. It follows that $a\cdot (b-c)=a\cdot b-a\cdot
c+a$, for all $a,b,c\in B$. For $a\in B$, we denote by $\lambda_a$
the map $B\longrightarrow B$ defined by $\lambda_a(b)=a\cdot b-a$,
for all $b\in B$. In fact $\lambda_a\in \aut(B,+)$, and
$\lambda\colon (B,\cdot)\longrightarrow \aut(B,+)$, defined by
$\lambda(a)=\lambda_a$, is a  group homomorphism (see \cite{CJO2}).
The socle, $\soc(B)$, of a left brace $B$ is defined as
$$\soc(B)=\{ a\in B\mid \lambda_a=\id_B\}.$$
It is an  ideal of $B$, i.e. a normal subgroup of $(B,\cdot )$ that
is invariant under all maps $\lambda_a$. (For the definitions of a
homomorphism of left braces, of a right brace and of related notions
we refer to \cite{CJO2}). Note that the maps $\lambda_a$ give a
useful link between the two operations in a left brace $B$, that is
$$a\cdot b=a+\lambda_a(b)\quad\mbox{ and }\quad a+b=a\cdot \lambda^{-1}_a(b),$$
for all $a,b\in B$. By a subgroup of a left brace $B$ we mean a
subgroup of the multiplicative group of $B$.

Given a solution $(X,r)$ of the YBE, the groups $G(X,r)$ and
$\mathcal{G}(X,r)$  each have  a natural left brace  structure. The
additive group of $G(X,r)$ is the free abelian group with basis $X$
and $\lambda_x(y)=\sigma_x(y)$, for all $x,y\in X\subseteq G(X,r)$.
Furthermore, the map $x\mapsto \sigma_x$ extends to an onto
(multiplicative) group homomorphism
$$\phi\colon G(X,r)\longrightarrow
\mathcal{G}(X,r)$$  and $\ker(\phi)=\soc(G(X,r))$ is an ideal of the
left brace $G(X,r)$. Hence
$$\mathcal{G}(X,r)\cong G(X,r)/\soc(G(X,r))$$
has a natural induced left brace structure (see also
\cite[Section~3]{Ta} and \cite[Sections~3 and~5]{GI15}). It follows
that $\phi : G(X,r)\longrightarrow \mathcal{G}(X,r)$ is a
homomorphism of left braces and, for every $g\in G(X,r)$, the map
$\phi(g)$ is the restriction of $\lambda_g$ to $X$. In particular,
$\phi(a+b)=\phi(a)+\phi(b)$, where the latter is the sum taken in
the brace $\mathcal{G}(X,r)$.

\begin{lemma}\label{socle}
Let $(X,r)$ be a solution of the YBE such that $\ret(X,r)=(X,r)$.
Then $\soc(\mathcal{G}(X,r))=\{ 1\}$.
\end{lemma}
\begin{proof}
We have that $r(x,y)=(\sigma_x(y),\sigma^{-1}_{\sigma_x(y)}(x))$,
for some $\sigma_x\in \sym(X)$.  Let $g\in \soc(\mathcal{G}(X,r))$.
Then there exist $x_1,\dots ,x_n\in X$ and
$\varepsilon_1,\dots,\varepsilon_n\in\{ -1,1\}$ such that
$g=\sigma_{x_1}^{\varepsilon_1}\cdots\sigma_{x_n}^{\varepsilon_n}$
and $gh-g=h$ for all $h\in \mathcal{G}(X,r)$. In particular, for all
$z\in X$  we  have
\begin{eqnarray*}
\sigma_z&=&g\sigma_z-g\\
&=&\sigma_{x_1}^{\varepsilon_1}\cdots\sigma_{x_n}^{\varepsilon_n}\sigma_z-\sigma_{x_1}^{\varepsilon_1}\cdots\sigma_{x_n}^{\varepsilon_n}\\
&=&\phi(x_1^{\varepsilon_1}\cdots
x_n^{\varepsilon_n}z-x_1^{\varepsilon_1}\cdots
x_n^{\varepsilon_n})\\
&=&\phi(\lambda_{x_1^{\varepsilon_1}\cdots
x_n^{\varepsilon_n}}(z))\\
&=&\sigma_{\lambda_{x_1^{\varepsilon_1}\cdots
x_n^{\varepsilon_n}}(z)}.
\end{eqnarray*}
Since, by assumption, $\ret(X,r)=(X,r)$ we get that
$\lambda_{x_1^{\varepsilon_1}\cdots x_n^{\varepsilon_n}}(z)=z$, for
all $z\in X$. Thus
$$g=\sigma_{x_1}^{\varepsilon_1}\cdots\sigma_{x_n}^{\varepsilon_n}=\phi(x_1^{\varepsilon_1}\cdots x_n^{\varepsilon_n})=\id_X,$$
and therefore $\soc(\mathcal{G}(X,r))=\{ 1\}$.
\end{proof}

Clearly, the converse of this result is not true. For example, let
$X=\{ 1,2\}$ and let $r\colon X^2\longrightarrow X^2$ be defined by
$r(x,y)=(y,x)$. Then $(X,r)$ is a solution of the YBE,
$\ret(X,r)\neq (X,r)$  and $\soc(\mathcal{G}(X,r))=\{ 1\}$.  What is
true is the following:
\begin{remark}{\rm
If $B$ is a left brace with $\soc(B)=\{1\}$, then there exists a
solution $(X,r)$ such that $\mathcal{G}(X,r)\cong B$ as left braces,
and $\ret(X,r)=(X,r)$.  Indeed, consider the associated solution of
$B$: $X=B$ and the map $r$ is given by
$$
\begin{array}{cccc}
r:& B\times B& \longrightarrow& B\times B\\
 &  (a,b)& \mapsto& (\lambda_a(b),\lambda^{-1}_{\lambda_a(b)}(a)).
\end{array}
$$
Note that,
$$\mathcal{G}(X,r)=\langle \lambda_a\mid a\in
B\rangle=\{ \lambda_a\mid a\in B\}\cong B/\soc(B)=B.$$ Moreover,
$\ret(X,r)= (X,r)$, because if $\lambda_{a_1}=\lambda_{a_2}$, then
$\lambda_{a_2^{-1}a_1}=\id$, and since the socle is trivial,
$a^{-1}_2a_1=1$.}
\end{remark}

If $B$ is a finite non-trivial two-sided brace, then $\soc(B)\neq\{
1\}$ \cite[Proposition 3]{CJO2}. By Lemma \ref{socle}, any such
solution $(X,r)$ with $\mathcal{G}(X,r)\cong B$ satisfies
$\ret(X,r)\neq(X,r)$. In fact, $(X,r)$ is a multipermutation
solution, (see \cite[Corollary~5.17]{GI15} or the proof of
\cite[Theorem 3]{CJO2} and the comments after this proof). Hence to
study finite non-multipermutation solutions, one should consider
only finite left braces which are not two-sided.

\begin{lemma}\label{orbits}
Let $(X,r)$ be a solution of the YBE and let $\{ X_i\}_{i\in I}$ be
the family of all orbits of $X$ under the action of
$\mathcal{G}(X,r)$. Suppose $\leq$ is a well-order on $I$. For each
$i\in I$ denote by $G_i$ the subgroup of $G(X,r)$ generated by
$X_i$. Then
\begin{itemize}
\item[(i)] $G_i$ is a subbrace of $G(X,r)$, invariant by the action of
$\mathcal{G}(X,r)$.
\item[(ii)]  $G_iG_j=G_jG_i$ for all $i,j\in I$.
\item[(iii)] Every $g\in G(X,r)\setminus \{ 1\}$ has a unique presentation as a product $g=g_1\cdots g_m$,
where $g_j\in G_{i_j}\setminus\{ 1\}$, for $j=1,\dots ,m$, and
$i_1<\dots <i_m$  are elements of $I$. Moreover, $g$ can be
presented uniquely as a sum $g=h_1+\cdots+h_m$, where $h_j\in
G_{i_j}\setminus\{ 1\}$, for $j=1,\dots ,m$, and $i_1<\dots <i_m$
are elements of $I$.
\end{itemize}
\end{lemma}

\begin{proof}
Let $G_i^+$ be the additive subgroup of $G(X,r)$ generated by $X_i$.
We shall prove that $G_i^+=G_i$. Let $g\in G_i$. Then there exist
$x_1,\dots ,x_k\in X_i$ and integers $\varepsilon_1,\dots,
\varepsilon_k\in\{ -1,1\}$ such that $g=x_1^{\varepsilon_1}\cdots
x_k^{\varepsilon_k}$. We have
\begin{eqnarray*}
g&=&x_1^{\varepsilon_1}\cdots x_k^{\varepsilon_k}\\
&=&x_1^{\varepsilon_1}\cdots x_{k-1}^{\varepsilon_{k-1}}+\lambda_{x_1^{\varepsilon_1}\cdots x_{k-1}^{\varepsilon_{k-1}}}(x_k^{\varepsilon_k})\\
&=&x_1^{\varepsilon_1}+\lambda_{x_1^{\varepsilon_1}}(x_2^{\varepsilon_2})+\lambda_{x_1^{\varepsilon_1}x_2^{\varepsilon_2}}(x_3^{\varepsilon_3})+\dots
+\lambda_{x_1^{\varepsilon_1}\cdots
x_{k-1}^{\varepsilon_{k-1}}}(x_k^{\varepsilon_k}).
\end{eqnarray*}
Since $x_i^{-1}=-\lambda_{x_i^{-1}}(x_i)$, it is clear that
$G_i\subseteq G_i^+$. Let $h\in G_i^+$. Then there exist $y_1,\dots
y_t\in X_i$ and integers $\nu_1,\dots, \nu_t\in\{ -1,1\}$ such that
$h=\nu_1y_1+\dots +\nu_ty_t$.  We have
\begin{eqnarray*}
h&=&\nu_1y_1+ \dots +\nu_ty_t\\
&=&(\nu_1y_1+\dots +\nu_{t-1}y_{t-1})\lambda^{-1}_{\nu_1y_1+\dots +\nu_{t-1}y_{t-1}}(\nu_t y_t)\\
&=&(\nu_1y_1)\lambda^{-1}_{\nu_1y_1}(\nu_2
y_2)\lambda^{-1}_{\nu_1y_1+\nu_2y_2}(\nu_3y_3)\cdots\lambda^{-1}_{\nu_1y_1+\dots
+\nu_{t-1}y_{t-1}}(\nu_t y_t).
\end{eqnarray*}
Since $-y_i=(\lambda^{-1}_{-y_i}(y_i))^{-1}$, it is clear that
$G_i^+\subseteq G_i$. Hence $G_i^+=G_i$. Therefore $G_i$ is a
subbrace of $G(X,r)$ and clearly it is invariant  under the action
of $\mathcal{G}(X,r)$. This proves $(i)$.

By $(i)$, we know that $\lambda_{g}(G_i)=G_i$, for all $g\in G(X,r)$
and $i\in I$. Let $g\in G_i$ and $h\in G_j$. By
\cite[Lemma~2(i)]{CJO2},
$gh=\lambda_g(h)\lambda^{-1}_{\lambda_g(h)}(g)\in G_jG_i$. Therefore
$G_iG_j\subseteq G_jG_i$. Thus $(ii)$ follows by symmetry.

Therefore, for every $g\in G(X,r)\setminus \{ 1\}$, there exist a
positive integer $m$, $i_1,\dots ,i_m\in I$ and $g_j\in G_{i_j}$,
for $j=1,\dots ,m$, such that $i_1<\dots <i_m$ and $g=g_1\cdots
g_m$. Suppose that $g_1\cdots g_m=g'_1\cdots g'_m$, for $g_j,g'_j\in
G_{i_j}$. Then,  by the above,
\begin{eqnarray*}g_1^{-1}g'_1&=&g_2\cdots g_m(g'_m)^{-1}\cdots
(g'_2)^{-1}\\
&=&g_2+\lambda_{g_2}(g_3)+\dots +\lambda_{g_2\cdots
g_{m-1}}(g_m)+\lambda_{g_2\cdots g_{m}}((g'_m)^{-1})\\
&&+\dots + \lambda_{g_2\cdots g_m(g'_m)^{-1}\cdots
(g'_3)^{-1}}((g'_2)^{-1})\in G_1^+\cap (G_2^++\dots +G_m^+).
\end{eqnarray*}
Since the additive group of $G(X,r)$ is free abelian with basis $X$,
we have that $G_1^+\cap (G_2^++\dots +G_m^+)=\{ 0\}$. Hence
$g_1=g'_1$. By induction on $m$, it follows that $g_j=g'_j$ for all
$j=1,\dots ,m$.  Let $h_1=g_1$ and $h_{i}=\lambda_{g_1\cdots
g_{i-1}}(g_i)$, for $1<i\leq m$. We have $g=h_1+\cdots+h_m$ and
$h_j\in G_{i_j}\setminus\{ 1\}$, for $j=1,\dots ,m$. Therefore,
$(iii)$ follows.
\end{proof}


Let $(X,r)$ be a solution of the YBE. We know that $G(X,r)$ is a
group presented with the set of generators $X$ and with relations
$xy=zt$ whenever $r(x,y)=(z,t)$.  Since the relations are
homogeneous, the group $G(X,r)$ has a degree function $\deg\colon
G(X,r)\longrightarrow \Z$ such that $\deg(x)=1$, for all $x\in X$.
Therefore, for $x_1,x_2,\dots ,x_m\in X$ and $n_1,n_2, \dots
,n_m\in \Z$, $\deg(x_1^{n_1}\cdots x_m^{n_m})=\sum_{l=1}^{m}n_l$.

\begin{remark}\label{degsum}{\rm
Since $(G(X,r),+)\cong \Z^{(X)}$, we can also define an additive
degree function  $\deg_+$ as follows $\deg_+(n_1x_1+\cdots +n_m
x_m)=\sum_{l=1}^{m}n_l$,  for $x_1,\dots ,x_m\in X$ and $n_1,\dots
,n_m\in \Z$. In fact, the two functions coincide.

Indeed, take an arbitrary element $g=x_1^{\varepsilon_1}\cdots
x_k^{\varepsilon_k}$ of $G(X,r)$, where $x_1,\dots x_k\in X$ and
$\varepsilon_1,\dots, \varepsilon_k\in\{ -1,1\}$. Note that
$\deg(g)=\sum_{i=1}^k \varepsilon_i$. As seen before, in a left
brace, we can pass from the multiplicative form to the additive form
through the lambda maps, and this in the following way:
\begin{eqnarray*}
g&=&x_1^{\varepsilon_1}\cdots x_k^{\varepsilon_k}\\
&=&x_1^{\varepsilon_1}+\lambda_{x_1^{\varepsilon_1}}(x_2^{\varepsilon_2})+\lambda_{x_1^{\varepsilon_1}x_2^{\varepsilon_2}}(x_3^{\varepsilon_3})+\dots
+\lambda_{x_1^{\varepsilon_1}\cdots
x_{k-1}^{\varepsilon_{k-1}}}(x_k^{\varepsilon_k}).
\end{eqnarray*}
Since $x_i^{-1}=-\lambda_{x_i^{-1}}(x_i)$ and $\lambda_h(x)\in X$
for any $h\in G(X,r)$ and any $x\in X$, it is clear that there exist
$y_1,\dots , y_k\in X$ such that
$$g=\varepsilon_1y_1+\dots +\varepsilon_ky_k.$$
 So we get $\deg_{+}(g)=\sum_{i=1}^k \varepsilon_i=\deg(g)$,
as claimed.

Now that we know that $\deg_+=\deg$, we can prove some properties of this function:
for any $g,h\in G(X,r)$,
\begin{enumerate}[(a)]
\item $\deg(\lambda_g(h))=\deg(h)$. We use the additive definition $\deg_+$
of the function.
Assume $h=\varepsilon_1 x_1+\cdots +\varepsilon_k x_k$, where
$x_1,\dots,x_k\in X$ and
$\varepsilon_1,\dots,\varepsilon_k\in\{-1,1\}$. Then
$\lambda_g(h)=\varepsilon_1 \lambda_g(x_1)+\cdots +\varepsilon_k
\lambda_g(x_k)$, and since $\lambda_g(x)\in X$ for any $g\in G$
and $x\in X$, $\deg(\lambda_g(h))=
\deg_+(\lambda_g(h))=\sum_{i=1}^k\varepsilon_i=\deg(h)$.
\item $\deg(g+h)=\deg(g)+\deg(h)$. This is a direct consequence of
the application of the additive function $\deg_+$.
\item $\deg(g\cdot h)=\deg(g)+\deg(h)$. This follows from the definition of $\deg$.
\end{enumerate}}
\end{remark}

 We will use all these facts in the proof of the
next lemma.

\begin{lemma}\label{ideal}
Let $(X,r)$ be a solution of the YBE. Let $\{ X_i\}_{i\in I}$ be the
family of all orbits of $X$ under the action of $\mathcal{G}(X,r)$.
Let $G_i$ be the subgroup of $G(X,r)$ generated by $X_i$. Let $\leq$
be a well-order on $I$. Let $H=\{ g_1\cdots g_n\mid g_l\in G_{i_l}$,
for $i_1<i_2<\dots <i_n$ in $I$ and $\deg(g_l)=0 \}=\{ g_1+\cdots +g_n\mid g_l\in G_{i_l}$,
for $i_1<i_2<\dots <i_n$ in $I$ and $\deg(g_l)=0 \}$. Then $H$ is an
ideal of the left brace  $G(X,r)$.
\end{lemma}

\begin{proof}
The set $\{ g_1\cdots g_n\mid g_l\in G_{i_l}$, for $i_1<i_2<\dots
<i_n$ in $I$ and $\deg(g_l)=0 \}$ is equal to $\{ g_1+\cdots
+g_n\mid g_l\in G_{i_l}$, for $i_1<i_2<\dots <i_n$ in $I$ and
$\deg(g_l)=0 \}$ because of Lemma~\ref{orbits} and the fact that the
multiplicative and the additive degree coincide. Hence $H$ is
well-defined.

By the definition of $H$ in multiplicative terms and Lemma~\ref{orbits}, it is easy to see
that $H$ is a normal subgroup of $G(X,r)$. Let $h\in H$ and $g\in
G(X,r)$. There exist elements $i_1<i_2<\dots <i_n$ in $I$ and
$g_l\in G_{i_l}$ such that $h=g_1+\cdots +g_n$ and $\deg(g_l)=0$ for
$l=1,2,\dots, n$. We have that
\begin{eqnarray*}
\lambda_g(h)&=&\lambda_g(g_1+\cdots +g_n)\\
&=&\lambda_g(g_1)+\lambda_{g}(g_2)+\dots +\lambda_{g}(g_n).
\end{eqnarray*}
Now $\lambda_g(g_l)\in G_l$ and $\deg(\lambda_g(g_l))=\deg(g_l)=0$,
so $H$ is $\lambda_g$-invariant and the assertion follows.
\end{proof}

\section{The main construction}\label{constr}
In this section we construct a new family of irretractable
square-free solutions of the YBE. These will be strong twisted
unions of multipermutation solutions of multipermutation level
$2$.

Strong twisted unions of solutions of the YBE were introduced in
\cite[Definition~5.1]{gat-maj}. In fact, the original definition
only covered the union of two quadratic sets. The general definition
appeared later in \cite[Definition~3.5]{GIC}. Recall that a solution
$(X,r)$ of the YBE is a strong twisted union of a set of solutions
$\{(X_j,r_j)\mid j\in J\}$, with $1<|J|$, if the sets $X_j$ are
$\mathcal{G}(X,r)$-invariant subsets of $X$, $X=\bigcup_{j\in
J}X_j$, $X_j\cap X_k=\emptyset$ for $j\neq k$, $r_j$ is the
restriction of $r$ to $X_j^2$ and, for all $j,k\in J$ such that
$j\neq k$,
\begin{eqnarray}\label{strong}
&&\sigma_{\gamma_x(z)}(y)=\sigma_{z}(y)\quad\mbox{and}\quad
\gamma_{\sigma_z(x)}(t)=\gamma_{x}(t),\end{eqnarray}
for all $x,y\in
X_j$, $z,t\in X_k$, where $r(a,b)=(\sigma_a(b),\gamma_b(a))$, for
all $a,b\in X$.

Let $A$ and $B$ be a nontrivial (additive) abelian groups. Let $I$
be a set with $|I|>1$ and let $X(A,B,I)=A\times B\times I$. Let
$\varphi_1\colon A\longrightarrow B$ be a map of sets such that
$\varphi_1(-a)=\varphi_1(a)$ for all $a\in A$. Let $\varphi_2\colon
B\longrightarrow A$ be a homomorphism of groups. For $a\in A$, $b\in
B$ and $i\in I$, let $\sigma_{(a,b,i)}\colon X(A,B,I)\longrightarrow
X(A,B,I)$ be the map defined by
$$\sigma_{(a,b,i)}(c,d,j)=\left\{\begin{array}{ll}
(c,d+\varphi_1(a-c),j)&\mbox{ if }i=j,\\
(c+\varphi_2(b),d,j)&\mbox{ if }i\neq j,
\end{array}\right.$$
for all $c\in A$, $d\in B$ and $j\in I$. Note that
$\sigma_{(a,b,i)}$ is bijective and
$$\sigma^{-1}_{(a,b,i)}(c,d,j)=\left\{\begin{array}{ll}
(c,d-\varphi_1(a-c),j)&\mbox{ if }i=j,\\
(c-\varphi_2(b),d,j)&\mbox{ if }i\neq j,
\end{array}\right.$$
for all $c\in A$, $d\in B$ and $j\in I$.

Let $r\colon X(A,B,I)^2\longrightarrow X(A,B,I)^2$ be defined by
$$r((a,b,i),(c,d,j))=(\sigma_{(a,b,i)}(c,d,j),
\sigma^{-1}_{\sigma_{(a,b,i)}(c,d,j)}(a,b,i)).$$ For $i\in I$, put
$X_i=A\times B\times\{ i\}$. Clearly we have that
$r^2=\id_{X(A,B,I)^2}$, i.e. $r$ is involutive.

\begin{lemma}\label{new} For $c\in A$, $d\in B$ and $j\in
I$, let $\gamma_{(c,d,j)}\colon X(A,B,I)\longrightarrow X(A,B,I)$ be
the map defined by
$$\gamma_{(c,d,j)}(a,b,i)=\sigma^{-1}_{\sigma_{(a,b,i)}(c,d,j)}(a,b,i),$$
for all $a\in A$, $b\in B$ and $i\in I$. Then $\gamma_{(c,d,j)}$ is
bijective and
$$\gamma^{-1}_{(c,d,j)}=\sigma_{(c,d,j)}.$$
\end{lemma}
\begin{proof}
Note that
$$\sigma^{-1}_{\sigma_{(a,b,i)}(c,d,j)}(a,b,i)=\left\{\begin{array}{ll}
(a,b-\varphi_1(c-a),i)&\mbox{ if }i=j,\\
(a-\varphi_2(d),b,i)&\mbox{ if }i\neq j.
\end{array}\right.$$
Therefore the result follows.
\end{proof}

\begin{remark}\label{referee}{\rm
Lemma~\ref{new} means that $(X(A,B,I),r)$ satisfies condition {\bf
lri} (see \cite[Definition~2.6]{GI15}). By \cite[Fact~2.7]{GI15}
every square-free solution of the Yang-Baxter equation satisfies
condition {\bf lri}. But the converse is not true. By
\cite[Fact~2.8]{GI15}, since $(X(A,B,I),r)$ is involutive and
satisfies {\bf lri}, we have that $(X(A,B,I),r)$ also is cyclic (see
\cite[Definition~2.6]{GI15}). In \cite{GI15} Gateva-Ivanova
continued her systematic study of solutions of the Yang-Baxter
equation with these important conditions.}
\end{remark}

\begin{theorem}\label{solution} $(X(A,B,I),r)$ is a solution of the YBE
and the following conditions hold.
\begin{itemize}
\item[(i)] $(X(A,B,I),r)$ is square-free if and only if $\varphi_1(0)=0$.
\item[(ii)] If $\varphi_1^{-1}(0)=\{0\}$ and $\varphi_2$ is injective, then $(X(A,B,I),r)$ is irretractable.
\item[(iii)] Every $X_i$ is
invariant under the action of $\mathcal{G}(X(A,B,I),r)$, and if
$r_i$ is the restriction of $r$ to $X_i^2$, then $(X_i,r_i)$ is a
multipermutation solution of multipermutation level at most two and
$(X(A,B,I), r)$ is a strong twisted union of the solutions
$(X_i,r_i)$.
\item[(iv)] If $\varphi_1(A)$ generates $B$ as a group and $\varphi_2$ is surjective, then
the orbits for the action of $\mathcal{G}(X(A,B,I),r)$ on $X(A,B,I)$
are $X_i$, for $i\in I$.
\end{itemize}
\end{theorem}
\begin{proof} We know that $r$ is involutive. By Lemma~\ref{new}, $r$ is non-degenerate.
Furthermore
\begin{eqnarray}\label{newdef}
&& r((a,b,i),(c,d,j))=(\sigma_{(a,b,i)}(c,d,j),
\sigma^{-1}_{(c,d,j)}(a,b,i)).
\end{eqnarray}
By \cite[Proposition~2]{CJO2}, to prove that $(X(A,B,I),r)$ is a
solution of the YBE it is enough to check that
$$\sigma_{(a,b,i)}\sigma_{\sigma^{-1}_{(a,b,i)}(c,d,j)}=
\sigma_{(c,d,j)}\sigma_{\sigma^{-1}_{(c,d,j)}(a,b,i)},$$ for all
$a,c\in A$, $b,d\in B$ and $i,j\in I$. Since
$\varphi_1(-a)=\varphi_1(a)$ and $\varphi_2$ is a homomorphism of
groups, these equalities follow at once from the following two
formulas, where $e\in A$, $f\in B$ and $k\in I$.
\begin{eqnarray*}
\lefteqn{\sigma_{(a,b,i)}\sigma_{\sigma^{-1}_{(a,b,i)}(c,d,j)}(e,f,k)}\\
&=&
\left\{\begin{array}{ll}
(e,f+\varphi_1(c-e)+\varphi_1(a-e),k)&\mbox{ if
}i=j=k,\\
(e+\varphi_2(d-\varphi_1(a-c))+\varphi_2(b),f,k)&\mbox{ if }i=j\neq
k,\\
(e+\varphi_2(b),f+\varphi_1(c-\varphi_2(b)-e),k)&\mbox{ if }i\neq
j=k,\\
(e+\varphi_2(d),f+\varphi_1(a-e-\varphi_2(d)),k)&\mbox{ if
}j\neq i=k,\\
(e+\varphi_2(d)+\varphi_2(b),f,k)&\mbox{ if }j\neq i, i\neq k\\
&\mbox{ and } j\neq k,
\end{array}\right.
\end{eqnarray*}
and
\begin{eqnarray*}
\lefteqn{\sigma_{(c,d,j)}\sigma_{\sigma^{-1}_{(c,d,j)}(a,b,i)}(e,f,k)}\\
&=&
\left\{\begin{array}{ll}
(e,f+\varphi_1(a-e)+\varphi_1(c-e),k)&\mbox{ if
}i=j=k,\\
(e+\varphi_2(b-\varphi_1(c-a))+\varphi_2(d),f,k)&\mbox{ if }i=j\neq
k,\\
(e+\varphi_2(b),f+\varphi_1(c-e-\varphi_2(b)),k)&\mbox{ if }i\neq
j=k,\\
(e+\varphi_2(d),f+\varphi_1(a-\varphi_2(d)-e),k)&\mbox{ if
}j\neq i=k,\\
(e+\varphi_2(b)+\varphi_2(d),f,k)&\mbox{ if }j\neq i, i\neq k\\
&\mbox{ and } j\neq k.
\end{array}\right.
\end{eqnarray*}
Therefore, indeed $(X(A,B,I),r)$ is a solution of the YBE.

$(i)$ Note that $\sigma_{(a,b,i)}(a,b,i)=(a,b+\varphi_1(0),i)$.
Hence $(X(A,B,I),r)$ is square-free if and only if $\varphi_1(0)=0$.

$(ii)$ Suppose that $\varphi_1^{-1}(0)=\{ 0\}$ and that $\varphi_2$
is injective. Let $(a,b,i),(c,d,j)\in X(A,B,I)$ be two distinct
elements. If $i\neq j$, then
$\sigma_{(a,b,i)}(e,f,i)=(e,f+\varphi_1(a-e),i)$ and
$\sigma_{(c,d,j)}(e,f,i)=(e+\varphi_2(d),f,i)$. Since $A\neq \{
0\}$, we can take $e\in A$ such that $\varphi_1(a-e)\neq 0$.
Therefore, if $i\neq j$, then $\sigma_{(a,b,i)}\neq
\sigma_{(c,d,j)}$. Suppose that $i=j$. Then $(a,b)\neq (c,d)$. If
$a\neq c$, then $\sigma_{(a,b,i)}(c,d,i)=(c,d+\varphi_1(a-c),i)$ and
$\sigma_{(c,d,j)}(c,d,i)=(c,d,i)$. Thus, in this case, since
$\varphi_1(a-c)\neq 0$, $\sigma_{(a,b,i)}\neq \sigma_{(c,d,j)}$.
Suppose $i=j$ and $a=c$. Then $b\neq d$, and for $k\in I\setminus\{
i\}$ there are equalities
$\sigma_{(a,b,i)}(0,0,k)=(\varphi_2(b),0,k)$ and
$\sigma_{(c,d,j)}(0,0,k)=(\varphi_2(d),0,k)$, which by the
injectivity of $\varphi_2$, imply again $\sigma_{(a,b,i)}\neq
\sigma_{(c,d,j)}$. Thus we have shown that $\sigma_{(a,b,i)}=
\sigma_{(c,d,j)}$ if and only if $(a,b,i)=(c,d,j)$.  Therefore
$\ret(X(A,B,I),r)=(X(A,B,I),r)$.

$(iii)$ It follows from the definition of $r$ that each $X_i$ is
invariant under the action of $\mathcal{G}(X(A,B,I),r)$. Let
$\sigma_{(a,b,i)}|_{X_i}$ be the restriction of $\sigma_{(a,b,i)}$
to the set $X_i=A\times B\times \{ i\}$. It is easy to check that
$$\sigma_{(a,b,i)}|_{X_i}=\sigma_{(c,d,i)}|_{X_i}$$
 if $a=c$. Since
$$r((a,b,i),(c,d,i))=((c,d+\varphi_1(a-c),i),(a,b-\varphi_1(c-a),i)),$$
one easily gets that $\ret(X_i,r_i)$ is  a trivial solution and thus
$\ret^2(X_i,r_i)$ is a solution on a set of cardinality $1$.
Therefore $(X_i,r_i)$ is a multipermutation solution of
multipermutation level  at most two. Let $i,j$ be distinct elements
in $I$. To show that $(X(A,B,I),r)$ is a strong twisted union of the
solutions $(X_i,r_i)$, in view of (\ref{strong}) and (\ref{newdef}),
we should check that
$$\sigma_{\sigma^{-1}_{(a,b,i)}(c,d,j)}(e,f,i)=\sigma_{(c,d,j)}(e,f,i)$$
and
$$\sigma^{-1}_{\sigma_{(e,f,j)}(c,d,i)}(a,b,j)=\sigma^{-1}_{(c,d,i)}(a,b,j),$$
for all $a,c,e\in A$ and $b,d,f\in B$. We have
\begin{eqnarray*}
\sigma_{\sigma^{-1}_{(a,b,i)}(c,d,j)}(e,f,i)&=&\sigma_{(c-\varphi_2(b),d,j)}(e,f,i)\\
&=&(e+\varphi_2(d),f,i)\\
&=&\sigma_{(c,d,j)}(e,f,i), \quad\mbox{and}\\
\sigma^{-1}_{\sigma_{(e,f,j)}(c,d,i)}(a,b,j)&=&\sigma^{-1}_{(c+\varphi_2(f),d,i)}(a,b,j)\\
&=&(a-\varphi_2(d),b,j)\\
&=&\sigma^{-1}_{(c,d,i)}(a,b,j).
\end{eqnarray*}
Hence, indeed, $(X(A,B,I),r)$ is a strong twisted union of the
solutions $(X_i,r_i)$.

$(iv)$ Suppose that $\varphi_1(A)$ generates $B$ and that
$\varphi_2$ is surjective. Let $a\in A$ and $b\in B$. There exist
$a_1,\dots ,a_s\in A$, $d\in B$ and $z_1,\dots ,z_n\in \mathbb{Z}$
such that $b=z_1\varphi_1(a_1)+\dots +z_s\varphi_1(a_s)$ and
$\varphi_2(d)=a$. Note that if $i\neq k$, then
$\sigma_{(0,d,i)}(0,0,k)=(a,0,k)$ and
$\sigma^{z_1}_{(a_1+a,0,k)}\cdots\sigma^{z_s}_{(a_s+a,0,k)}(a,0,k)=(a,b,k)$.
Hence the orbit of $(0,0,k)$ is $A\times B\times\{ k\}$, and this
finishes the proof.
\end{proof}

\begin{remark}  \label{vendramin}
{\rm For $A=B=\Z/(2)$ and $I=\{1,2\}$ the solution $(X(A,B,I),r)$ of
the above theorem, with $\varphi_1=\varphi_2=\id_A$, is isomorphic
to the solution of \cite[Example~3.9]{Ven}. Recall that two
solutions $(X,r)$ and $(X',r')$ of the YBE are isomorphic if there
exists a bijective map $\eta:X\longrightarrow X'$ such that
$$r'(\eta(x),\eta(y))=(\eta(\sigma_x(y)),\eta(\gamma_y(x))),$$
where $r(x,y)=(\sigma_x(y),\gamma_y(x))$, for $x,y\in X$.}
\end{remark}

Recall that Gateva-Ivanova conjectured that every finite square-free
solution of the YBE is a multipermutation solution \cite[Strong
Conjecture 2.28(I)]{GI}. The  construction of Vendramin, given in
Remark~\ref{vendramin}, was the first counterexample to this
conjecture. In fact, he constructed a family of counterexamples
consisting of extensions of the one given above, i.e. square-free
solutions $(Y,s)$ such that there exists a surjective homomorphism
$Y\longrightarrow X$ of solutions, where $X$ is the solution
$(X(A,B,I),r)$ given in Remark~\ref{vendramin}.
 Notice
that \cite[Lemma 4]{CJO2} implies that these solutions are not
multipermutation solutions. However, we do not know whether these
solutions are irretractable.

Perhaps Gateva-Ivanova expected that her conjecture might be too
strong, because  the following related question was formulated in
\cite{GIC}.
\bigskip

\noindent {\bf Question 1.} \cite[Open Questions~6.13 (II)(4)]{GIC}
Let $(X,r)$ be a square-free solution of the YBE. Suppose that
$(X,r)$ is a strong twisted union of two solutions of the YBE. Is it
true that $(X,r)$ is a multipermutation solution?
\bigskip

Notice also that there are examples of square-free
multipermutation solutions of the YBE which are not a strong
twisted union of two solutions of the YBE, see
\cite[Theorem~3.1]{CJO}.

\begin{remark}
{\rm The solutions $(X(A,B,I),r)$,  with $\varphi_1^{-1}(0)=\{0\}$
and $\varphi_2$ injective, are new counterexamples to \cite[Strong
Conjecture 2.28(I)]{GI}, and  if moreover $|I|=2$, then they answer
in the negative Question 1. Note that if $\varphi_1^{-1}(0)=\{ 0\}$,
then $\sigma_{(0,0,i)}|_{X_i}\neq \sigma_{(a,0,i)}|_{X_i}$ for all
$a\neq 0$, thus, in this case, the multipermutation level of
$(X_i,r_i)$ is $2$. The solution constructed by Vendramin
\cite[Example 3.6]{Ven} also gives a negative answer to Question 1,
but this fact is not noticed in \cite{Ven}.}
\end{remark}

\section{The permutation group and the structure group of $(X(A,B,I),r)$}

First we will study the structure of the multiplicative group of the
left brace $\mathcal{G}(X(A,B,I),r)$.

Let $G$ and $H$ be two abelian groups, and let $W$ be the set of all
functions $f\colon H\longrightarrow G$. Then $W$ is an abelian group
with the sum $f_1+f_2$ defined by $(f_1+f_2)(h)=f_1(h)+f_2(h)$, for
$f_1,f_2\in W$ and $h\in H$. Recall that the complete wreath product
$G\bar\wr H$ can be defined as the semidirect product $W\rtimes H$
with respect to the action of $H$ on $W$ defined by
$(hf)(x)=f(x-h)$, for $f\in W$ and $h\in H$. The wreath product
$G\wr H$ is defined similarly, but replacing $W$ by the abelian
group $W'$ of all functions $f\colon H\longrightarrow G$ such that
the set $\{h\in H:f(h)\neq 0\}$ is finite. Obviously, when $H$ is
finite, the complete wreath product and the wreath product coincide.

\begin{proposition}\label{permutation}
The permutation group $\mathcal{G}(X(A,B,I),r)$ is isomorphic to a
subgroup of the Cartesian product of $|I|$ copies of the complete
wreath product $\langle \varphi_1(A)\rangle\bar\wr A$. In
particular, if moreover $A$ and $B$ are finite abelian $p$-groups,
and $I$ is finite, then $\mathcal{G}(X(A,B,I),r)$ is a finite
$p$-group.
\end{proposition}

\begin{proof}
Let $B_1=\langle\varphi_1(A)\rangle$. Let $W$ be the set of all
functions $f:A\longrightarrow B_1$. Consider the set $S=\{
\sigma_{(a,b,i)}\mid a\in A,\; b\in B,\; i\in I\}$. Denote the
Cartesian product of $|I|$ copies of $W\rtimes A$ by $(W\rtimes
A)^{I}$ and its elements by $((f_j,a_j))_{j\in I}$, with $f_j\in W$
and $a_j\in A$. For each $a\in A$, let $f_a\in W$ denote the map
defined by $f_a(x)=\varphi_1(a-x)$, for all $x\in A$. We define a
map $\nu\colon S\longrightarrow (W\rtimes A)^{I}$ by
$\nu(\sigma_{(a,b,i)})=((f_j,a_j))_{j\in I}$, where
$$f_j=\left\{ \begin{array}{ll}
f_a&\mbox{ if }j=i,\\
0&\mbox{ if } j\in I\setminus\{ i\},
\end{array} \right. \quad \mbox{and}\quad a_j=\left\{ \begin{array}{ll}
0&\mbox{ if }j=i,\\
\varphi_2(b)&\mbox{ if } j\in I\setminus\{ i\}.
\end{array} \right.$$

We claim that $\nu$ can be extended to an injective homomorphism
$$\nu\colon\mathcal{G}(X(A,B,I),r)\longrightarrow (W\rtimes A)^{I}.$$
Let $a_1,\dots ,a_r\in A$, $b_1,\dots, b_r\in B$, $i_1,\dots ,i_r\in
I$ and $\varepsilon_1,\dots,\varepsilon_r\in \{ -1,1\}$. To prove
the claim it is enough to prove that
$$\sigma_{(a_1,b_1,i_1)}^{\varepsilon_1}\cdots \sigma_{(a_r,b_r,i_r)}^{\varepsilon_r}=\id_{X(A,B,I)}$$
if and only if
$$((f_{r,j},a_{r,j})^{-\varepsilon_r})_{j\in I}\cdots ((f_{1,j},a_{1,j})^{-\varepsilon_1})_{j\in
I}=((0,0))_{j\in I},$$ where
$$f_{k,j}=\left\{ \begin{array}{ll}
f_{a_k}&\mbox{ if }j=i_k,\\
0&\mbox{ if } j\in I\setminus\{ i_k\},
\end{array} \right. \quad \mbox{and}\quad a_{k,j}=\left\{ \begin{array}{ll}
0&\mbox{ if }j=i_k,\\
\varphi_2(b_k)&\mbox{ if } j\in I\setminus\{ i_k\}.
\end{array} \right.$$
We know that
\begin{eqnarray*}
\lefteqn{((f_{r,j},a_{r,j})^{-\varepsilon_r})_{j\in I}\cdots
((f_{1,j},a_{1,j})^{-\varepsilon_1})_{j\in I}}\\
&=&\left(\left(\left(\frac{-\varepsilon_r-1}{2}a_{r,j}\right)(-\varepsilon_rf_{r,j}),-\varepsilon_ra_{r,j}\right)\right)_{j\in
I}\\
&&\hphantom{xxxxxxxxx}\cdots
\left(\left(\left(\frac{-\varepsilon_1-1}{2}a_{1,j}\right)(-\varepsilon_1f_{1,j}),-\varepsilon_1a_{1,j}\right)\right)_{j\in I}\\
&=&\left(\left(\sum_{l=1}^r\left(\sum_{l<k\leq
r}-\varepsilon_ka_{k,j}\right)\left(\left(\frac{-\varepsilon_l-1}{2}a_{l,j}\right)(-\varepsilon_lf_{l,j})\right),
-\sum_{k=1}^{r}\varepsilon_ka_{k,j}\right)\right)_{j\in I}.
\end{eqnarray*}
On the other hand
\begin{eqnarray*}\lefteqn{\sigma_{(a_1,b_1,i_1)}^{\varepsilon_1}\cdots
\sigma_{(a_r,b_r,i_r)}^{\varepsilon_r}(x,y,j)}\\
&=&\sigma_{(a_1,b_1,i_1)}^{\varepsilon_1}\cdots
\sigma_{(a_{r-1},b_{r-1},i_{r-1})}^{\varepsilon_{r-1}}(x+(1-\delta_{i_r,j})\varepsilon_r\varphi_2(b_r),\\
&&\hphantom{xxxxxxxxxxxxxxxxxxxxxxx}y+\delta_{i_r,j}\varepsilon_r\varphi_1(a_r-x),j)\\
&=&\sigma_{(a_1,b_1,i_1)}^{\varepsilon_1}\cdots
\sigma_{(a_{r-2},b_{r-2},i_{r-2})}^{\varepsilon_{r-2}}\left(x+\sum_{k=r-1}^r(1-\delta_{i_k,j})\varepsilon_k\varphi_2(b_k)\right.,\\
&&\hphantom{xxxxxxxx}\left.y+\sum_{l=r-1}^r\delta_{i_l,j}\varepsilon_l\varphi_1\left(a_l-x-\sum_{l<k\leq r}(1-\delta_{i_k,j})\varepsilon_k\varphi_2(b_k)\right),j\right)\\
&=& \\ &\vdots & \\ &=&\left(x+\sum_{k=1}^r(1-\delta_{i_k,j})\varepsilon_k\varphi_2(b_k),\right.\\
&&\quad\left.
y+\sum_{l=1}^r\delta_{i_l,j}\varepsilon_l\varphi_1\left(a_l-x-\sum_{l<k\leq
r}(1-\delta_{i_k,j})\varepsilon_k\varphi_2(b_k)\right),j\right),
\end{eqnarray*}
where $\delta_{i,j}$ is the Kronecker delta, that is
$$\delta_{i,j}=\left\{
\begin{array}{ll}
1&\mbox{ if } i=j,\\
0&\mbox{ if } i\neq j.
\end{array}\right.$$
Hence
$$\sigma_{(a_1,b_1,i_1)}^{\varepsilon_1}\cdots \sigma_{(a_r,b_r,i_r)}^{\varepsilon_r}=\id_{X(A,B,I)}$$
if and only if
$$\sum_{k=1}^r(1-\delta_{i_k,j})\varepsilon_k\varphi_2(b_k)=0$$
and
$$\sum_{l=1}^r\delta_{i_l,j}\varepsilon_l\varphi_1\left(a_l-x-\sum_{l<k\leq
r}(1-\delta_{i_k,j})\varepsilon_k\varphi_2(b_k)\right)=0,$$ for all
$j\in I$ and all $x\in A$. Note that
$$\sum_{k=1}^{r}\varepsilon_ka_{k,j}=\sum_{k=1}^r(1-\delta_{i_k,j})\varepsilon_k\varphi_2(b_k)$$
and
\begin{eqnarray*}\lefteqn{\left( \sum_{l=1}^r\left(\sum_{l<k\leq r}-\varepsilon_ka_{k,j}\right)\left(\left(\frac{-\varepsilon_l-1}{2}a_{l,j}\right)(-\varepsilon_lf_{l,j})\right)\right)(x)}\\
&=&-\sum_{l=1}^r\varepsilon_lf_{l,j}\left(x-\frac{-\varepsilon_l-1}{2}a_{l,j}-\sum_{l<k\leq
r}-\varepsilon_ka_{k,j}\right)\\
&=&-\sum_{l=1}^r\delta_{i_l,j}\varepsilon_l\varphi_1\left(a_l-x-\frac{\varepsilon_l+1}{2}a_{l,j}-\sum_{l<k\leq
r}\varepsilon_ka_{k,j}\right)\\
&=&-\sum_{l=1}^r\delta_{i_l,j}\varepsilon_l\varphi_1\left(a_l-x-(1-\delta_{i_l,j})\frac{\varepsilon_l+1}{2}\varphi_2(b_l)\right.\\
&&\hphantom{xxxxxxxxx}-\left.\sum_{l<k\leq
r}(1-\delta_{i_k,j})\varepsilon_k\varphi_2(b_{k})\right)\\
&=&-\sum_{l=1}^r\delta_{i_l,j}\varepsilon_l\varphi_1\left(a_l-x-\sum_{l<k\leq
r}(1-\delta_{i_k,j})\varepsilon_k\varphi_2(b_{k})\right),
\end{eqnarray*}
where, in the last equality, the term
$(1-\delta_{i_l,j})\frac{\varepsilon_l+1}{2}\varphi_2(b_l)$
disappears because, when $(1-\delta_{i_l,j})=1$, then
$\delta_{i_l,j}=0$, whence the term
$$\delta_{i_l,j}\varepsilon_l\varphi_1\left(a_l-x-(1-\delta_{i_l,j})\frac{\varepsilon_l+1}{2}\varphi_2(b_l)-\sum_{l<k\leq
r}(1-\delta_{i_k,j})\varepsilon_k\varphi_2(b_{k})\right)$$ becomes
zero, and does not appear in the sum. Therefore the claim is proved.
\end{proof}

With some additional hypothesis, we can determine precisely the
permutation group of $(X(A,B,I),r)$. Note that the next result gives
examples of solutions of the YBE with permutation group of
arbitrarily large nilpotency class.

\begin{corollary}\label{cyclic}
Assume that $A=B=\Z/(k)$, $k>1$, $I$ is a finite set such that
$\gcd(|I|-1,k)=1$, $\varphi_2$ is surjective, $\varphi_1(0)=0$ and
$\varphi_1(x)=1$ for any $x\in A\setminus\{0\}$. Then,
$\mathcal{G}(X(A,B,I),r)\cong (A\wr A)^{\mid I\mid}$.

In this case, the derived length of $\mathcal{G}(X(A,B,I),r)$ is
$2$. The permutation group $\mathcal{G}(X(A,B,I),r)$ is nilpotent if
and only if $k=p^\alpha$ for some prime $p$, and, in this case, its
nilpotency class is equal to $(\alpha (p-1)+1)p^{\alpha-1}.$ In
particular, if $A$ is of prime order $p$ then the nilpotency class
is $p$.
\end{corollary}
\begin{proof}
First, observe that $\langle\varphi_1(A)\rangle=A$.  By
Proposition~\ref{permutation} (and its proof), we know that
$\nu:\mathcal{G}(X(A,B,I),r)\longrightarrow (A\wr A)^{|I|}$ is
injective. So,  in order to prove the first claim it is enough to
show that $\nu$ is surjective. As before, denote by $f_{a}$, $a\in
A$, the map given by $f_{a}(c)=\varphi_1(a-c)$, for $c\in A$. First
we prove that, by definition of $\varphi_1$, the set $\{f_{a}: a\in
A\}$ generates the abelian group $W$ consisting of the  maps from
$A$ to $A$. Indeed, if $f$ is any map from $A$ to itself, we have to
find $z_{a}\in \Z$, $a\in A$, such that $f(c)=\sum_{a\in A} z_{a}
f_{a}(c)$ for any $c\in A$. Observe that $\sum_{a\in A} z_{a}
f_{a}(c)=\sum_{a\in A\setminus\{c\}} z_a$ by the definition of
$\varphi_1$. This is a system of linear equations with coefficients
in $\Z/(k)$ with $k$ equations in $k$ variables. The associated
matrix of the system is
$$
N_k=
\begin{pmatrix}
0& 1& \cdots & 1\\
1& 0& \ddots& \vdots\\
\vdots& \ddots& \ddots& 1\\
1& \cdots& 1& 0\\
\end{pmatrix}
\in M_k(\Z/(k)).
$$
One can prove that $\det(N_k)=(-1)^{k-1}(k-1)$, which is invertible in $\Z/(k)$,
so the system has a solution.

Second, we prove that $S=\{s_j(f,b):j\in I,b\in A, f\in W\}$, where
$s_j(f,b)=((f_i,b_i))_{i\in I}$ is an element of $(A\wr A)^{\mid I\mid}$ defined by
$$(f_{i},b_i)=\left\{ \begin{array}{ll}
(f,0)&\mbox{ if }i=j,\\
(0,b)&\mbox{ if } i\in I\setminus\{ j\},
\end{array} \right.$$
is a set of generators of $(A\wr A)^{\mid I\mid}$. Given an
arbitrary element $((f'_i,c_i))_{i\in I}$ of $(A\bar\wr A)^{\mid
I\mid}$, consider the equations $\sum_{i\in I\setminus\{j\}}
x_i=c_j$, one for each $j\in I$, in the variables $\{x_i\}_{i\in
I}$. Then this is a system of $| I|$ linear equations in $| I|$
variables, and its associated matrix is $N_{| I|}$. We know that
$\det(N_{\mid I\mid})=(-1)^{\mid I\mid-1}(| I|-1)$, so the system
has a solution since, by the hypothesis, $\gcd(|I|-1,k)=1$. Thus
there exist elements $\{b_i\}_{i\in I}$ in $A$ such that
$(\sum_{i\in I} b_i)-b_j=c_j$ for any $j\in I$. Define functions
$f_j$, $j\in I$, by $f_j(c)=f'_j(c+\sum_{1\leq k<j} b_k)$ for any
$c\in A$. Then,
\begin{eqnarray*}
\lefteqn{s_1(f_1,b_1)\cdot s_2(f_2,b_2)\cdots s_{|I|}(f_{|I|},b_{|I|})}\\
&=&((f_1,0),(0,b_1),\dots,(0,b_1))\cdot((0,b_2),(f_2,0),(0,b_2),\dots,(0,b_2))\cdots \\
&&\lefteqn{\cdot((0,b_{|I|}),\dots,(0,b_{|I|}),(f_{|I|},0))}\\
&=&\left((f_1,\sum_{i\in I} b_i-b_1),(b_1f_2,\sum_{i\in I} b_i-b_2),\dots,\right.\\
&&\hphantom{xxxxxx}\left.((b_1+\cdots +b_{|I|-1})f_{|I|},\sum_{i\in I} b_i-b_{|I|})\right)\\
&=&((f'_i,c_i))_{i\in I},
\end{eqnarray*}
showing that $S$ is a set of generators of $(A\wr A)^{\mid I\mid}$,
as claimed.

Thus, to prove that $\nu$ is surjective, it is enough to show that,
for any $s_i(f,b)$ in $S$, there exists a $\tau\in
\mathcal{G}(X(A,B,I),r)$ such that $\nu(\tau)=s_i(f,b)$. By what we
checked at the beginning of this proof, there exist integers $z_a$,
$a\in A$, such that $f=\sum_{a\in A} z_a f_{a}$. Assume
$A=\{a_1,\dots,a_k\}$ with $a_1=0$, and denote $z_i=z_{a_i}$. On the
other hand, choose $a\in A$ such that $\varphi_2(a)=b$ (using here
that $\varphi_2$ is surjective). Then, define the element
$$
\tau=\sigma_{(0,0,i)}^{z_1-1}\sigma_{(a_2,0,i)}^{z_2}\cdots\sigma_{(a_k,0,i)}^{z_k}\sigma_{(0,a,i)}\in
\mathcal{G}(X(A,B,I),r).
$$
Note that, by definition of $\nu$, $\nu(\sigma_{(a,b,i)})=s_i(f_a,\varphi_2(b))$.
Moreover, $s_i(f,a)\cdot s_i(g,b)=s_i(f+g,a+b)$.
These two facts explain the following computation:
\begin{eqnarray*}
\nu(\tau)&=&\lefteqn{\nu(\sigma_{(0,0,i)})^{z_1-1}\cdot\nu(\sigma_{(a_2,0,i)})^{z_2}\cdots\nu(\sigma_{(a_k,0,i)})^{z_k}\cdot\nu(\sigma_{(0,a,i)})}\\
&=&s_i(f_0,0)^{z_1-1}\cdot s_i(f_{a_2},0)^{z_2}\cdots s_i(f_{a_k},0)^{z_k}\cdot s_i(f_0,\varphi_2(a))\\
&=&s_i\left((z_1-1)f_0+\sum_{i=2}^{k}z_{i}f_{a_i}+f_0,\varphi_2(a)\right)\\
&=&s_i(f,b).
\end{eqnarray*}
Hence, this finally shows that $\nu$ is surjective.

At this point, we have proved that $\mathcal{G}(X(A,B,I),r)\cong
(A\wr A)^{\mid I\mid}$. Observe that $A\wr A$ is a semidirect
product of two abelian groups, so its derived length is $2$. It
follows that $(A\wr A)^{\mid I\mid}$ also has  derived length equal
to $2$. Concerning the nilpotency of $\mathcal{G}(X(A,B,I),r)$,
recall that a direct product of groups $G\times H$ is nilpotent if
and only if $G$ and $H$ are nilpotent. Besides, the results of
\cite{baumslag} imply that a wreath product of two finite groups $G$
and $H$ is nilpotent if and only if $G$ and $H$ are $p$-groups for
the same prime $p$. So, in our case, $(A\wr A)^{\mid I\mid}$ is
nilpotent if and only if $A=\Z/(p^\alpha)$ for some prime $p$.

Moreover, it is possible to compute the nilpotency class of this
wreath product. We use the following known result (see
\cite[Theorem 5.1]{liebeck}): the nilpotency class of $G\wr H$,
where $G$ and $H$ are finite abelian $p$-groups such that $G$ has
exponent $p^n$ and $H\cong
\Z/(p^{\beta_1})\times\cdots\times\Z/(p^{\beta_m})$ with
$\beta_1\geq \cdots\geq \beta_m$, is equal to
$$
(n-1)(p-1)p^{\beta_1-1}+1+\sum_{i=1}^m (p^{\beta_i}-1).
$$
Applying this result to $G=H=A=\Z/(p^\alpha)$, we get that the
nilpotency class of $A\wr A$ is equal to
$$
(\alpha p-\alpha+1)p^{\alpha-1}.
$$
The nilpotency class of $(A\wr A)^{\mid I\mid}$ is also equal to $(\alpha p-\alpha+1)p^{\alpha-1}$
because the class of a direct product $G\times H$ is equal to the maximum of the class of $G$ and the class of $H$.

Observe that, in particular, for $A=\Z/(p)$, we obtain nilpotency
class equal to $p$.
\end{proof}

By Lemma~\ref{socle}, we know that $\soc(\mathcal{G}(X(A,B,I),r))=\{
1\}$ if $(X(A,B,I),r)$ is irretractable. By Theorem~\ref{solution},
this happens if $\varphi_1^{-1}(0)=\{ 0\}$ and $\varphi_2$ is
injective. In view of the strategy explained in the introduction, it
would be interesting to know  whether under some conditions on $A$,
$B$, $I$, $\varphi_1$ and $\varphi_2$, the left brace
$\mathcal{G}(X(A,B,I),r)$ can be simple.

We do not know any new simple left brace of the form
$\mathcal{G}(X(A,B,I),r)$ and it seems difficult to study the ideal
structure of $\mathcal{G}(X(A,B,I),r)$ in general. The next result
shows that there are some non-trivial ideals in some cases where the
socle is trivial. Although these left braces are not simple, maybe
they can be used to construct new families of simple left braces
with the techniques used in \cite{B3}, because they have trivial
socle.

\begin{proposition}
Assume that $A=B=\Z/(k)$, $k>1$, $I$ is a finite set such that
$\gcd(|I|-1,k)=1$, $\varphi_2$ is surjective, $\varphi_1(0)=0$ and
$\varphi_1(x)=1$ for any $x\in A\setminus\{0\}$. Then
$\mathcal{G}(X(A,B,I),r)$ is not a simple left brace.
\end{proposition}

\begin{proof}
Let $I=\{i_1,\dots ,i_n\}$ with $|I|=n$. Let $G_i$ be the subgroup
of the structure group $G(X(A,B,I),r)$ generated by $\{(a,b,i)\mid
a,b\in A\}$. By Theorem~\ref{constr}$(iv)$, the sets $\{(a,b,i)\mid
a,b\in A\}$ are the orbits of $X(A,B,I)$ under the action of
$\mathcal{G}(X(A,B,I),r)$. Thus, by Lemma~\ref{orbits}, every
element $g$ of $G(X(A,B,I),r)$ can be written uniquely as
$g=g_{1}\cdots g_{n}$, for some $g_{l}\in G_{i_l}$. Let $H=\{
g_1\cdots g_n\mid g_l\in G_{i_l} \mbox{ and } \deg(g_l)=0 \}$. By
Lemma~\ref{ideal}, $H$ is an ideal of  the left brace
$G(X(A,B,I),r)$.

Let $\phi\colon G(X(A,B,I),r)\longrightarrow
\mathcal{G}(X(A,B,I),r)$ be the natural map. We shall prove that
$\phi(H)$ is a non-trivial proper ideal of the left brace
$\mathcal{G}(X(A,B,I),r)$.

Note that
$$\sigma_{(a,b,i)}^{k}(x,y,j)=(x+(1-\delta_{i,j})k\varphi_2(b),y+\delta_{i,j}k\varphi_1(a-x),j)=(x,y,j),$$
for all $a,b,x,y\in A$ and all $i,j\in I$. Hence
$\phi((a,b,j)^k)=\sigma_{(a,b,i)}^k=\id$. We claim that
\begin{eqnarray}\label{hker}
&&\ker(\phi)H=\mathrm{gr}((a,b,i)^k\mid a,b\in A,\; i\in I)H.
\end{eqnarray}
It is clear that $\mathrm{gr}((a,b,i)^k\mid a,b\in A,\; i\in
I)H\subseteq \ker(\phi)H$. Let $g\in \ker(\phi)$. By
Lemma~\ref{orbits}, there exist unique $g_1,\dots, g_l\in
G(X(A,B,I),r)$ such that $g=g_1\cdots g_n$ and $g_l\in G_{i_l}$, for
$l=1,\dots ,n$. Note that
\begin{eqnarray*}
\lefteqn{\sigma_{(a,b,i)}\sigma_{(c,d,i)}(x,y,j)}\\
&=&(x+(1-\delta_{i,j})\varphi_2(d+b),y+\delta_{i,j}(\varphi_1(c-x)+\varphi_1(a-x)),j)\\
&=&\sigma_{(c,d,i)}\sigma_{(a,b,i)}(x,y,j).
\end{eqnarray*}
Therefore $(a,b,i)^{-1}(c,d,i)^{-1}(a,b,i)(c,d,i)\in
\ker(\phi)\cap H$, for all $a,b,c,d\in A$ and all $i\in I$. Hence,
to prove that $g\in \mathrm{gr}((a,b,i)^k\mid a,b\in A,\; i\in
I)H$, we may assume that every $g_l$ is of the form
\begin{eqnarray*}
g_l&=&(0,0,i_l)^{z_{0,0,l}}(0,1,i_l)^{z_{0,1,l}}\cdots
(0,k-1,i_l)^{z_{0,k-1,l}}\\
&&\; (1,0,i_l)^{z_{1,0,l}}(1,1,i_l)^{z_{1,1,l}}\cdots
(1,k-1,i_l)^{z_{1,k-1,l}}\\
&&\; \cdots (k-1,0,i_l)^{z_{k-1,0,l}}(k-1,1,i_l)^{z_{k-1,1,l}}\cdots
(k-1,k-1,i_l)^{z_{k-1,k-1,l}}.
\end{eqnarray*}
Hence
$$\id=\phi(g)=\left(\prod_{p,q=0}^{k-1}\sigma_{(p,q,i_1)}^{z_{p,q,1}}\right)
\cdot
\left(\prod_{p,q=0}^{k-1}\sigma_{(p,q,i_2)}^{z_{p,q,2}}\right)\cdots
\left(\prod_{p,q=0}^{k-1}\sigma_{(p,q,i_n)}^{z_{p,q,n}}\right).$$
Therefore, for every $x,y\in A$ and $j\in I$, we have
\begin{eqnarray*}
(x,y,j)&=&\left(
x+\sum_{l=1}^n(1-\delta_{l,j})\varphi_2\left(\sum_{p,q=0}^{k-1}z_{p,q,l}q\right),\right.\\
&&\qquad\left. y+\sum_{p,q=0}^{k-1}z_{p,q,j}\varphi_1\left(
p-x-\sum_{l=j+1}^{n}\varphi_2\left(\sum_{p',q'=0}^{k-1}z_{p',q',l}q'\right)\right),j\right).
\end{eqnarray*}
Thus
$$\sum_{l=1}^n(1-\delta_{l,j})\varphi_2\left(\sum_{p,q=0}^{k-1}z_{p,q,l}q\right)=0,$$
for all $j\in I$. Since $A$ is finite and $\varphi_2$ is a
surjective endomorphism of $A$, we have that $\varphi_2$ is an
automorphism of $A$, and
$$\sum_{l=1}^n(1-\delta_{l,j})\sum_{p,q=0}^{k-1}z_{p,q,l}q=0,$$
for all $j\in I$. Note that the system of linear equations
$$\sum_{l=1}^n(1-\delta_{l,j})x_l=0,\quad \mbox{ for }j\in I$$
over $A$ has only the trivial solution $x_l=0$ for all $l$, because
$\gcd(|I|-1,k)=1$. Therefore
$$\sum_{p,q=0}^{k-1}z_{p,q,l}q=0,$$
for all $l$. On the other hand, we also have that
$$\sum_{p,q=0}^{k-1}z_{p,q,j}\varphi_1\left(
p-x-\sum_{l=j+1}^{n}\varphi_2\left(\sum_{p',q'=0}^{k-1}z_{p',q',l}q'\right)\right)=0,$$
for all $x\in A$ and all $j\in I$. Therefore
$$\sum_{p,q=0}^{k-1}z_{p,q,j}\varphi_1(p-x)=0,$$
for all  $x\in A$ and all $j\in I$. Since
$\varphi_1(p-x)=1-\delta_{p,x}$, and the system of linear equations
$$\sum_{p=0}^{k-1}(1-\delta_{p,x})t_p=0\quad\mbox{ for } x\in A$$
over $A$ has only the trivial solution $t_p=0$ for all $p\in A$, we
get that
$$\sum_{q=0}^{k-1}z_{p,q,j}=0\in A,$$
for all $p\in A$ and all $j\in I$. Hence  $\deg
(g_l)=\sum_{p,q=0}^{k-1}z_{p,q,l}=kz_l$ for some integer $z_l$. Thus
$(0,0,i_l)^{-kz_l}g_l\in H$, for all $l$. Hence
$$g=g_1\cdots g_n\in \mathrm{gr}((a,b,i)^k\mid a,b\in A,\; i\in
I)H$$ and this proves the claim (\ref{hker}). Now it is clear that
$\sigma_{(0,0,i_1)}\notin\phi(H)$. Hence $\phi(H)$ is a proper ideal
of $\mathcal{G}(X(A,B,I),r)$. Since
$$\sigma_{(1,0,i_1)}^{-1}\sigma_{(0,0,i_1)}(0,0,i_1)=\sigma_{(1,0,i_1)}^{-1}(0,0,i_1)=(0,1,i_1)\neq
(0,0,i_1)$$  and $\id\neq
\sigma_{(1,0,i_1)}^{-1}\sigma_{(0,0,i_1)}\in \phi(H)$, the result
follows.
\end{proof}

Consider now the structure group $G(X(A,B,I),r)$ of the solution
$(X(A,B,I),r)$ of Theorem~\ref{solution}. Suppose that $A$, $B$ and
$I$ are finite. Hence $X(A,B,I)$ is finite and by
\cite[Corollary~8.2.7]{JObook}, $G(X(A,B,I),r)$ is  solvable and a
Bieberbach group, i.e. a finitely generated torsion-free
abelian-by-finite group, see \cite{charlap}. It would be interesting
to characterize when the structure group of a solution is
poly-(infinite cyclic). Recall that a multipermutation  solution of
the YBE on a finite set $X$ has a structure group that is
poly-(infinite cyclic), see \cite[Proposition 8.2.12]{JObook}. It
remains an open question whether the converse holds. Farkas in
\cite[Theorem~23]{F} showed that a Bieberbach group is
poly-(infinite cyclic) if and only if every non-trivial subgroup has
a non-trivial center. This is one of the key ingredients of the
proof of the following result. Another key ingredient is based on
the application of the natural structure of a left brace on the
structure group of a solution of the YBE, explained in
Section~\ref{prelim}.

\begin{theorem}\label{poly}
If $A$, $B$ and $I$ are finite, $\varphi_1(0)=0$, $\varphi_1(A)$
generates $B$ and $\varphi_2$ is an isomorphism, then
$G(X(A,B,I),r)$ is not a poly-(infinite cyclic) group.
\end{theorem}
\begin{proof}
 Let $G_i$ be the subgroup
of $G(X(A,B,I),r)$ generated by $\{(a,b,i)\mid a\in A,\; b\in B\}$.
By Theorem~\ref{constr}$(iv)$, the sets $\{(a,b,i)\mid a\in A,\;
b\in B\}$ are the orbits of $X(A,B,I)$ under the action of
$\mathcal{G}(X(A,B,I),r)$. Let $I=\{i_1,\dots ,i_n\}$ with $|I|=n$.
Thus, by Lemma~\ref{orbits}, every element $g$ of $G(X(A,B,I),r)$
can be written uniquely as $g=g_{1}+\cdots +g_{n}$, for some
$g_{l}\in G_{i_l}$. Let $H=\{ g_1+\cdots +g_n\mid g_l\in G_{i_l}
\mbox{ and } \deg(g_l)=0 \}$. By Lemma~\ref{ideal}, $H$ is an ideal
of $G(X(A,B,I),r)$ and it is easy to see that
$G(X(A,B,I),r)/H\cong\Z^n$.

To prove that $G(X(A,B,I),r)$ is not poly-(infinite cyclic), it is
sufficient to show that $Z(H)=\{ 1\}$ (see \cite[Theorem~23]{F}).

We will show now that indeed $Z(H)=\{ 1\}$. Suppose $h\in Z(H)$.
Then $\lambda_h$ has finite order, since $\mathcal{G}(X(A,B,I),r)$
is a finite group. Let $s$ be the order of $\lambda_h$. So $h^s\in
Z(H)$ and $\lambda_{h^s}=\id$. The group $G(X(A,B,I),r)$ is
torsion-free, therefore, replacing $h$ by $h^s$, we may assume that
$\lambda_h=\id$. Let $g\in H$. Then
$$\lambda_g(h)=gh-g=hg-g=\lambda_h(g)+h-g=g+h-g=h.$$

Let $h_l\in G_{i_l}$ be such that $h=h_1+\cdots +h_n$. Then
$\lambda_g(h)=\lambda_g(h_1)+\cdots +\lambda_g(h_n)$. By
Lemma~\ref{orbits}, $G_i$ is invariant under the action of
$\mathcal{G}(X(A,B,I),r)$. Hence $\lambda_g(h_l)\in G_{i_l}$ for all
$l$. Since $h=\lambda_g(h)$, comparing their decompositions as sums
of element of the subgroups $G_{i_l}$, we have, by
Lemma~\ref{orbits}, that $\lambda_g(h_1)=h_1$. We know that the
additive group of $G(X(A,B,I),r)$ is free abelian with basis
$X(A,B,I)$. Thus, we may assume that
\begin{eqnarray}  \label{h1}
&&h_1=n_1(a_1,b_1,i_1)+\dots + n_m(a_m,b_m,i_1),
\end{eqnarray}
where $(a_1,b_1,i_1),\dots ,(a_m,b_m,i_1)$ are $m$ distinct elements
of $X(A,B,I)$ and $\sum_{l=1}^{m}n_l$ $=0$. By the hypothesis,
$\varphi_1(A)$ generates $B$, so for every $l$ there exist
$c_1,\dots ,c_s\in A$ and $z_1,\dots z_s\in \mathbb{Z}$ such that
$b_l-b_1=z_1\varphi_1(c_1)+\dots +z_s\varphi_1(c_s)$. Let
\begin{eqnarray*}f&=&(a_l,b_l,i_1)^{-z_1-\ldots -z_s}(a_l+c_1,0,i_1)^{z_1}\cdots
(a_l+c_s,0,i_1)^{z_s}\\
&&\cdot
(0,0,i_2)^{-1}(0,\varphi_2^{-1}(a_l-a_1),i_2).\end{eqnarray*}
 We have that $f\in H$.  Hence, by the above, $\lambda_f(h_1)=h_1$ and
\begin{eqnarray*}
\lefteqn{\lambda_f((a_1,b_1,i_1))}\\
&=&\sigma_{(a_l,b_l,i_1)}^{-z_1-\ldots
-z_s}\sigma_{(a_l+c_1,0,i_1)}^{z_1}\cdots
\sigma_{(a_l+c_s,0,i_1)}^{z_s}
\sigma_{(0,0,i_2)}^{-1}\sigma_{(0,\varphi_2^{-1}(a_l-a_1),i_2)}((a_1,b_1,i_1))\\
&=&\sigma_{(a_l,b_l,i_1)}^{-z_1-\ldots -z_s}\sigma_{(a_l+c_1,0,i_1)}^{z_1}\cdots \sigma_{(a_l+c_s,0,i_1)}^{z_s}((a_l,b_1,i_1))\\
&=&\sigma_{(a_l,b_l,i_1)}^{-z_1-\ldots -z_s}((a_l,b_l,i_1))\\
&=&(a_l, b_l,i_1).
\end{eqnarray*}
Hence, from (\ref{h1}) we get that $n_1=n_l$, for all $l=1,\dots
,m$. Since $\sum_{l=1}^{m}n_l=0$, it follows that $n_l=0$ for all
$l$, and thus $h_1=1$. Similarly one can prove that $h_2=\dots
=h_n=1$, and therefore $h=1$, as desired.
\end{proof}

The assertion of Theorem~\ref{poly} also seems to be of interest
from the point of view of the Kaplansky conjecture on non-existence
of nontrivial units in group algebras of torsion-free groups. The
structure groups $G(X(A,B,I),r)$ provide some natural nontrivial
examples for testing this conjecture.


\section*{Acknowledgments}
The two first-named authors were partially
supported by the grants DGI MICIIN MTM2011-28992-C02-01, and MINECO
MTM2014-53644-P. The third author is supported in part by
Onderzoeksraad of Vrije Universiteit Brussel and Fonds voor
Wetenschappelijk Onderzoek (Belgium). The fourth author is supported
by the National Science Centre grant 2013/09/B/ST1/04408 (Poland).

\vspace{30pt}

 \noindent \begin{tabular}{llllllll}
 D. Bachiller && F. Ced\'o  \\
 Departament de Matem\`atiques &&  Departament de Matem\`atiques \\
 Universitat Aut\`onoma de Barcelona &&  Universitat Aut\`onoma de Barcelona  \\
08193 Bellaterra (Barcelona), Spain    && 08193 Bellaterra (Barcelona), Spain \\
 dbachiller@mat.uab.cat &&  cedo@mat.uab.cat\\
   &&   \\
E. Jespers && J. Okni\'{n}ski  \\ Department of Mathematics &&
Institute of Mathematics
\\  Vrije Universiteit Brussel && Warsaw University \\
Pleinlaan 2, 1050 Brussel, Belgium && Banacha 2, 02-097 Warsaw, Poland\\
efjesper@vub.ac.be&& okninski@mimuw.edu.pl
\end{tabular}

\end{document}